\pdfoutput=1
\documentclass[oneside,english]{amsart}
\usepackage[T1]{fontenc}
\usepackage[latin9]{inputenc}
\usepackage{geometry}
\usepackage{babel}
\usepackage{amsbsy}
\usepackage{amstext}
\usepackage{amsthm}
\usepackage{amssymb}
\usepackage{graphicx}
\usepackage[all]{xy}
\usepackage[unicode=true,pdfusetitle,
 bookmarks=true,bookmarksnumbered=false,bookmarksopen=false,
 breaklinks=false,pdfborder={0 0 1},backref=false,colorlinks=false]
 {hyperref}

\makeatletter
\numberwithin{equation}{section}
\numberwithin{figure}{section}
\theoremstyle{plain}
\newtheorem{thm}{\protect\theoremname}
\theoremstyle{plain}
\newtheorem{cor}[thm]{\protect\corollaryname}
\theoremstyle{definition}
\newtheorem{defn}[thm]{\protect\definitionname}
\theoremstyle{remark}
\newtheorem{rem}[thm]{\protect\remarkname}

\author{Sungkyung Kang}
\address{Institute of Mathematical Sciences, The Chinese University of Hong Kong, Shatin, N.T. Hong Kong}
\email{skkang@math.cuhk.edu.hk}

\subjclass[2010]{57M27}
\keywords{Link TQFT; Ribbon concordances}

\makeatother

\providecommand{\corollaryname}{Corollary}
\providecommand{\definitionname}{Definition}
\providecommand{\remarkname}{Remark}
\providecommand{\theoremname}{Theorem}

\begin{document}
\title{Link homology theories and ribbon concordances}
\begin{abstract}
It was recently proved by several authors that ribbon concordances
induce injective maps in knot Floer homology, Khovanov homology, and
the Heegaard Floer homology of the branched double cover. We give
a simple proof of a similar statement in a more general setting, which
includes knot Floer homology, Khovanov-Rozansky homologies, and all
conic strong Khovanov-Floer theories. This gives a philosophical answer
to the question of which aspects of a link TQFT make it injective
under ribbon concordances.
\end{abstract}

\maketitle

\section{Introduction}

Given two knots $K_{1}$ and $K_{2}$, smoothly embedded in $S^{3}$,
a (smooth) cobordism from $K_{1}$ to $K_{2}$ is a smoothly embedded
surface $S$ in $S^{3}\times I$ such that $\partial C=(K_{1}\times\{0\})\sqcup(K_{2}\times\{1\})$.
A connected cobordism is called a concordance its genus is zero. By
endowing $S$ with a Morse function, it is easy to see that every
knot (or, in general, link) cobordism consists of births, saddles,
and deaths; when $S$ is a concordance and there are no deaths needed
to construct $S$, we say that $S$ is a ribbon concordance.

Unlike the knot concordance relation, which is symmetric, having a
ribbon concordance from a knot to another is not a symmetric relation.
It is conjectured by Gordon\cite{gordon1981ribbon} that existence
of a ribbon concordance induces a partial order on each knot concordance
class. There are several results in that direction, starting with
Gordon's result\cite{gordon1981ribbon}, that if $C$ is a ribbon
concordance from $K_{1}$ to $K_{2}$, then the map $\pi_{1}(S^{3}\backslash K_{1})\rightarrow\pi_{1}((S^{3}\times I)\backslash C)$
is injective and $\pi_{1}(S^{3}\backslash K_{2})\rightarrow\pi_{1}((S^{3}\times I)\backslash C)$
is surjective. 

Recently, it was proved by Zemke\cite{zemke2019knot} that, after
endowing $C$ with a suitable decoration, the cobordism map 
\[
\hat{F}_{C}:\widehat{HFK}(K_{1})\rightarrow\widehat{HFK}(K_{2})
\]
 is injective, so that if there exists a ribbon concordance from $K_{1}$
to $K_{2}$ and also from $K_{2}$ to $K_{1}$, then $\widehat{HFK}(K_{1})\simeq\widehat{HFK}(K_{2})$
as bigraded vector spaces. This result was then extended to other
link homology theories. For example, Levine and Zemke\cite{levine2019khovanov}
proved that the map 
\[
\mathbf{Kh}(C):\mathbf{Kh}(K_{1})\rightarrow\mathbf{Kh}(K_{2})
\]
 is injective. Later, Lidman, Vela-Vick, and Wong \cite{lidman2019heegaard}
proved that the map 
\[
\hat{F}_{\Sigma(C)}:\widehat{HF}(\Sigma(K_{1}))\rightarrow\widehat{HF}(\Sigma(K_{2}))
\]
 is also injective, where $\Sigma(K_{i})$ is the branched double
cover of $S^{3}$ along $K_{i}$ and $\Sigma(C)$ is the branched
double cover of $S^{3}\times I$ along $C$. Note that $\Sigma(C)$
is a 4-dimensional smooth cobordism from the 3-manifold $\Sigma(K_{1})$
to $\Sigma(K_{2})$.

However, the proofs of the above results rely on some special properties
that the link homology theories $\widehat{HFL}$, $\mathbf{Kh}$,
and $\widehat{HF}\circ\Sigma$ have. The proof of injectivity for
$\widehat{HFK}$ depends on its generalization to a TQFT of null-homologous
links in 3-manifolds, and the proof for $\mathbf{Kh}$ uses the fact
that it satisfies the neck-cutting relation in the dotted cobordism
category. Furthermore, the proof for $\widehat{HF}\circ\Sigma$ uses
graph cobordisms, defined and studied originally by Zemke\cite{zemke2015graph}.

In this paper, we give a simple proof of the injectivity of maps induced
by ribbon concordance in a much more general setting. The link homology
theories that we can use are multiplicative link TQFTs which are either
associative or Khovanov-like, whose definitions will be given in the
next section. In particular, our main theorem is the following.
\begin{thm}
\label{MainThm}Let $C$ be a ribbon concordance and $F$ be a multiplicative
TQFT of oriented links in $S^{3}$, which is either associative or
Khovanov-like. Then $F(C)$ is injective, and $F(\bar{C})$ is its
left inverse.
\end{thm}

The notion of associativity and Khovanov-like-ness, together with
multiplicativity, is so general that they include all conic strong
Khovanov-Floer theories, defined in \cite{saltz2017strong}, which
is based on the definition of Khovanov-Floer theories in \cite{baldwin2019functoriality},
and all Khovanov-Rozansky homologies, which were first defined in
\cite{khovanov2004matrix}. This gives us the following corollaries.
\begin{cor}
\label{MainCor}Let $C$ be a ribbon concordance and $F$ be either
a conic strong Khovanov-Floer theory or Khovanov-Rozansky $\mathfrak{gl}(n)$-homology
for some $n\ge2$. Then $F(C)$ is injective.
\end{cor}

\begin{cor}
\label{MainCor2}Let $K_{1}$ and $K_{2}$ be knots, such that there
exists a ribbon concordance $C$ from $K_{1}$ to $K_{2}$, and $C^{\prime}$
from $K_{2}$ to $K_{1}$. Then for any link TQFT $F$ which is either
a conic strong Khovanov-Floer theory or a Khovanov-Rozansky homology,
we have $F(K_{1})\cong F(K_{2})$.
\end{cor}

Furthermore, we will observe that if $F$ is actually a $\mathbb{Z}$-graded
theory, and satisfies some nice properties, then our proof of injectivity
simplifies even more. 

As a topological application of our arguments, we will give a very
simple alternative proof of Zemke's result on knot Floer homology.
Then we will also give another proof of Lidman-Vela-Vick-Wong's result
on $\widehat{HF}\circ\Sigma$ and prove the following theorem, regarding
the deck transformation action, denoted as $\tau$, and the involution
introduced in \cite{hendricks2017involutive}, denoted as $\iota$,
on the hat-flavor Heegaard Floer homology of branched double covers.
\begin{thm}
\label{MainApp}Suppose that a knot $K_{0}$ is ribbon concordant
to $K_{1}$. Then the following statements hold.
\begin{itemize}
\item If $K_{1}$ is an odd torus knot, then the $\tau$-action on $\widehat{HF}(\Sigma(K_{0}))$
are trivial.
\item If $K_{1}$ is a Montesinos knot, then the $\tau$-action and the
$\iota$-action on $\widehat{HF}(\Sigma(K_{0}))$ coincide, and there
exists a $\tau$-invariant basis of $\widehat{HF}(\Sigma(K_{0}))$
with only one fixed basis element.
\end{itemize}
\end{thm}

\subsection*{Acknowledgement}

The author would like to specially thank Marco Marengon for suggesting
an idea for the proof of the nicely graded case. The author is also
grateful to Abhishek Mallick for helpful discussions, and Monica Jinwoo
Kang, Andras Juhasz and Robert Lipshitz for numerous helpful comments
on this paper. Finally, the author would also like to thank University
of Oregon and UCLA for their hospitality.

\section{Multiplicative link TQFTs and conic strong Khovanov-Floer theories}

\subsection{Multiplicativity of a TQFT of links in $S^{3}$}

Recall that a (vector space valued) TQFT $F$ of (oriented) links
in $S^{3}$ is a functor 
\[
F:\mathbf{Link}_{S^{3}}^{+}\rightarrow\mathbf{Vect}_{\mathbb{F}}.
\]
 Here, $\mathbf{Link}_{S^{3}}^{+}$ is the category whose objects
are oriented links in $S^{3}$ and morphisms are oriented link cobordisms,
and $\mathbf{Vect}_{\mathbb{F}}$ is the category of vector spaces
over a fixed coefficient field $\mathbb{F}$.

Suppose that a link $L$ can be written as a disjoint union $L=L_{1}\sqcup L_{2}$,
i.e. there exists a genus zero Heegaard splitting $S^{3}=D_{1}\cup_{S^{2}}D_{2}$
such that $L\cap S^{2}=\emptyset$, $L\cap D_{1}=L_{1}$, and $L\cap D_{2}=L_{2}$.
Then we usually expect $F(L)$ to split along the disjoint union as
follows:
\[
F(L)\simeq F(L_{1})\otimes_{\mathbb{F}}F(L_{2}).
\]
 But the multiplicativity that we want $F$ to satisfy is stronger
than having such an isomorphism. Suppose that we are given link cobordisms
\begin{align*}
S_{1} & :L_{1}\rightarrow L_{1}^{\prime},\\
S_{2} & :L_{2}\rightarrow L_{2}^{\prime},
\end{align*}
 and consider the links $L=L_{1}\sqcup L_{2}$ and $L^{\prime}=L_{1}^{\prime}\sqcup L_{2}^{\prime}$.
When there exist disjoint open balls $V_{1},V_{2}\subset S^{3}$,
satisfying $L_{1}\subset V_{1}$ and $L_{2}\subset V_{2}$, such that
$S_{1}\subset V_{1}\times I$ and $S_{2}\subset V_{2}\times I$, we
can form the disjoint union cobordism $S=S_{1}\sqcup S_{2}$. The
cobordism $S$ is then a link cobordism from $L$ to $L^{\prime}$.

Now we have three linear maps:
\begin{align*}
F(S_{1}) & :F(L_{1})\rightarrow F(L_{1}^{\prime}),\\
F(S_{2}) & :F(L_{2})\rightarrow F(L_{2}^{\prime}),\\
F(S) & :F(L)\rightarrow F(L^{\prime}).
\end{align*}
 Using these maps, we define the multiplicativity of $F$ as follows.
\begin{defn}
A TQFT of (oriented) links in $S^{3}$ is \textit{multiplicative}
if we have identifications 
\begin{align*}
F(L) & \simeq F(L_{1})\otimes F(L_{2}),\\
F(L^{\prime}) & \simeq F(L_{1}^{\prime})\otimes F(L_{2}^{\prime}),
\end{align*}
 such that $F(S)=F(S_{1})\otimes F(S_{2})$ is satisfied. 
\end{defn}

Unfortunately, multiplicativity is not enough to prove that all ribbon
concordance induce injective maps, so we need to introduce some additional
conditions on multiplicative link TQFTs.
\begin{defn}
A multiplicative TQFT $F$ of oriented links in $S^{3}$ is \textit{associative}
if for any link $L=L_{1}\sqcup L_{2}$ such that $L_{1},L_{2}$ are
contained in disjoint open balls $V_{1},V_{2}$ respectively, we have
an associated isomorphism 
\[
F(L)\xrightarrow{\cong}F(L_{1})\otimes F(L_{2})
\]
 which depends only on the choice of open balls $V_{1}$ and $V_{2}$,
and if we are given a link $L=L_{1}\sqcup L_{2}\sqcup L_{3}$, the
following diagram commutes.
\[
\xymatrix{F(L)\ar[r]^{\cong}\ar[d]^{\cong} & F(L_{1}\sqcup L_{2})\otimes F(L_{3})\ar[d]^{\cong}\\
F(L_{1})\otimes F(L_{2}\sqcup L_{3})\ar[r]^{\cong} & F(L_{1})\otimes F(L_{2})\otimes F(L_{3})
}
\]
\end{defn}

As we will see in the next section, associativity is enough to prove
that ribbon concordance maps are injective. However, even when we
are given with a link TQFT which is multiplicative but not associative,
we are still able to find another condition which is sufficient for
our goal. 

Recall that, if $F$ is a TQFT of (oriented) links in $S^{3}$, then
the $\mathbb{F}$-vector space $F(\text{unknot})$ comes with the
following operations:
\begin{align*}
\text{birth map }b: & \mathbb{F}\rightarrow F(\text{unknot}),\\
\text{death }\epsilon: & F(\text{unknot})\rightarrow\mathbb{F}.
\end{align*}
 Also, we call the element $b(1)\in F(\text{unknot})$ as the unit
and denote it as $u$. Note that, if $F$ is the Khovanov homology
functor $\mathbf{Kh}$, then $\epsilon(u)=0$ and $u$ spans the kernel
of $\epsilon$.
\begin{defn}
A multiplicative TQFT $F$ of oriented links in $S^{3}$ is \textit{Khovanov-like}
if the unit $u\in F(\text{unknot})$ spans the kernel of the counit
$\epsilon$.
\end{defn}

\subsection{Khovanov-Floer theories}

The notion of Khovanov-Floer theory first appeared in \cite{baldwin2019functoriality}.
In that paper, Baldwin, Hedden, and Lobb gave its definition as follows.
\begin{defn}
Let $V$ be a graded vector space. a $V$-complex is a pair $(C,q)$
where $C$ is a filtered chain complex and $q:V\rightarrow E_{2}(C)$
is an isomorphism. A map of $V$-complexes is a filtered chain map.
When a map $f$ of $V$-complexes induces the identity map between
the $E_{2}$ pages, we say that $f$ is a quasi-isomorphism.
\end{defn}

\begin{defn}
A Khovanov-Floer theory $\mathcal{A}$ is a rule which assigns to
every link diagram $D$ a quasi-isomorphism class $\mathcal{A}(D)$
of $\mathbf{Kh}(D)$-complexes which satisfies the following conditions.
\begin{itemize}
\item If $D^{\prime}$ is planar isotopic to $D$, then there is a morphism
$\mathcal{A}(D)\rightarrow\mathcal{A}(D^{\prime})$ which induces
the isotopy map $\mathbf{Kh}(D)\xrightarrow{\sim}\mathbf{Kh}(D^{\prime})$
on the $E_{2}$ page.
\item If $D^{\prime}$ is obtained from $D$ by a diagrammatic 1-handle
attachment, then there is a morphism $\mathcal{A}(D)\rightarrow\mathcal{A}(D^{\prime})$
which induces the cobordism map $\mathbf{Kh}(D)\rightarrow\mathbf{Kh}(D^{\prime})$
on the $E_{2}$ page.
\item For any diagrams $D,D^{\prime}$, we have a morphism $\mathcal{A}(D\sqcup D^{\prime})\rightarrow\mathcal{A}(D)\otimes\mathcal{A}(D^{\prime})$
which induces the standard isomorphism $\mathbf{Kh}(D\sqcup D^{\prime})\xrightarrow{\sim}\mathbf{Kh}(D)\otimes\mathbf{Kh}(D^{\prime})$
on the $E_{2}$ page.
\item If $D$ is a diagram of an unlink, then the spectral sequence $E_{2}(\mathcal{A}(D))\Rightarrow E_{\infty}(\mathcal{A}(D))$
degenerates on the $E_{2}$ page.
\end{itemize}
\end{defn}

Later, Saltz gave a definition of strong Khovanov-Floer theories in
the following way.
\begin{defn}
A strong Khovanov-Floer theory $\mathcal{K}$ is a rule which assigns
a link diagram $D$ and a collection of auxiliary data $A$ a filtered
chain complex $\mathcal{K}(D,A)$ satisfying the following conditions.
\begin{itemize}
\item For any two collections $A_{\alpha},A_{\beta}$ of auxiliary data,
there is a homotopy equivalence $a_{\alpha}^{\beta}:\mathcal{K}(D,A_{\alpha})\rightarrow\mathcal{K}(D,A_{\beta})$.
We write $\mathcal{K}(D)$ for the canonical representative of the
transitive system $\{\mathcal{K}(D,A_{\alpha}),a_{\alpha}^{\beta}\}$,
i.e. the limit of the diagram $\{\mathcal{K}(D,A_{\alpha}),a_{\alpha}^{\beta}\}$
in the homotopy category of chain complexes.
\item If $D$ is a crossingless diagram of the unknot, then $H_{\ast}(\mathcal{K}(D))\simeq\mathbf{Kh}(D)$.
\item For diagrams $D,D^{\prime}$, we have $\mathcal{K}(D\sqcup D^{\prime})\simeq\mathcal{K}(D)\otimes K(D^{\prime})$.
\end{itemize}
Furthermore, a strong Khovanov-Floer theory also assigns maps to diagrammatic
cobordisms with auxiliary data. Those maps should satisfy the following
conditions.
\begin{itemize}
\item If $D^{\prime}$ is obtained from $D$ by a diagrammatic handle attachment,
then there is a function 
\[
\phi:\{\text{auxiliary data for }D\}\rightarrow\{\text{auxiliary data for }D^{\prime}\}
\]
 and a map 
\[
\mathfrak{h}_{A_{\alpha},\phi(A_{\alpha}),B}:\mathcal{K}(D,A_{\alpha})\rightarrow\mathcal{K}(D^{\prime},\phi(A_{\alpha})
\]
 where $B$ is some additional auxiliary data. In addition, if the
domain of $\phi$ is empty, then its codomain is also empty. This
gives a well-defined map 
\[
\mathfrak{h}_{B}:\mathcal{K}(D)\rightarrow\mathcal{K}(D^{\prime})
\]
 for a fixed $B$. Furthermore, for any two sets $B,B^{\prime}$ of
additional auxiliary data, we have $\mathfrak{h}_{B}\simeq\mathfrak{h}_{B^{\prime}}$.
\item If $D$ is a crossingless diagram of the unknot, then $\mathcal{K}(D)$
is isomorphic to $\mathbf{Kh}(D)$ as Frobenius algebras.
\item If $D^{\prime}$ is obtained from $D$ by a planar isotopy, then $\mathcal{K}(D)\simeq\mathcal{K}(D^{\prime})$.
\item Let $D=D_{0}\sqcup D_{1}$, $D^{\prime}=D_{0}^{\prime}\sqcup D_{1}^{\prime}$,
and suppose that $\Sigma_{0},\Sigma_{1}$ are diagrammatic cobordisms
from $D_{0}$ to $D_{0}^{\prime}$ and $D_{1}$ to $D_{1}^{\prime}$,
respectively. Take the disjoint union $\Sigma=\Sigma_{0}\sqcup\Sigma_{1}$.
Then we have 
\[
\mathcal{K}(\Sigma)=\mathcal{K}(\Sigma_{0})\otimes\mathcal{K}(\Sigma_{1}).
\]
\item The handle attachment maps are invariant under swapping the order
of handle attachments with disjoint supports, and satisfies movie
move 15, as shown in Figure 2 of \cite{saltz2017strong}.
\end{itemize}
\end{defn}

Unfortunately, for a strong Khovanov-Floer theory to induce a TQFT
of links in $S^{3}$, we need one more condition.
\begin{defn}
A strong Khovanov-Floer theory $\mathcal{K}$ is \textit{conic} if
for any link diagram $D$ and any crossing $c$ of $D$, we have 
\[
\mathcal{K}(D)\simeq\text{Cone}(\mathcal{K}(D_{0})\xrightarrow{\mathfrak{h}_{\gamma_{c}}}\mathcal{K}(D_{1})),
\]
 where $D_{0}$ and $D_{1}$ are the 0-resolution and the 1-resolution
of $D$ at $c$ and $\mathfrak{h}_{\gamma_{c}}$ is the diagrammatic
handle attachment map at $c$.
\end{defn}

The notion of conic strong Khovanov-Floer theories is very general.
The following list of link homology theories are examples of conic
strong Khovanov-Floer theories. (Actually, all strong Khovanov-Floer
theories known up to now are conic!) 
\begin{itemize}
\item Khovanov homology of $L$\cite{khovanov1999categorification};
\item Heegaard Floer homology of $\Sigma(\text{unknot}\sqcup L)$\cite{ozsvath2005heegaard};
\item Unreduced singular instanton homology of $L$\cite{kronheimer2011khovanov};
\item Bar-Natan homology of $L$\cite{rasmussen2010khovanov};
\item Szabo homology of $L$\cite{szabo2010geometric}.
\end{itemize}
When $\mathcal{K}$ is a conic strong Khovanov-Floer theory, its homology
$K=H_{\ast}(\mathcal{K})$ is functorial under link cobordisms in
$S^{3}\times I$, and thus a Khovanov-like multiplivative TQFT, by
Theorem 5.9 of \cite{saltz2017strong}. Hence we see that our conditions
on link TQFTs are general enough to cover all strong Khovanov-Floer
theories. Actually, even more is true: all strong Khovanov-Floer theories
known up to now are associative. But it is not clear whether the same
should also be true for all strong theories.

In this paper, we will confuse Khovanov-Floer theories with their
homology, so that when we say that $F$ is a conic strong Khovanov-Floer
theory, we will actually mean that $F$ is the multiplicative link
TQFT which arises as the homology of a conic strong Khovanov-Floer
theory.

\section{Proof of Theorem \ref{MainThm}}

\subsection{\label{subsec-saddle}An alternative decomposition of a saddle followed
by the dual saddle}

Let a link $L$ and a framed simple arc $a$ inside $S^{3}$, where
the interior of $a$ is disjoint from $L$ and $\partial a\subset L$,
are given. Then we can perform a saddle move along $a$to $L$. In
terms of cobordisms in $S^{3}\times I$, this corresponds to attaching
a 1-handle; denote the saddle cobordism as $S_{a}$. Then its upside-down
cobordism $\overline{S_{a}}$ can be considered as performing a ``dual
saddle'' move, which is a saddle move along a dual arc $a^{\ast}$,
as drawn in the right of Figure \ref{Fig7}.

\begin{figure}
\resizebox{.5\textwidth}{!}{\includegraphics{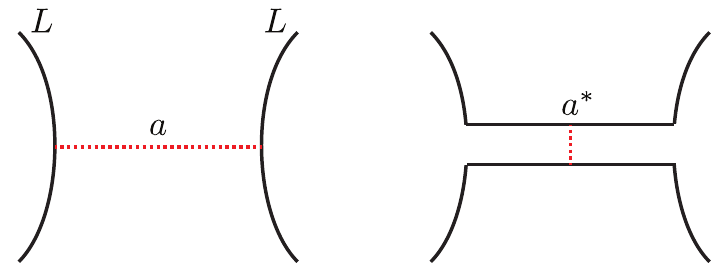}}
\caption{\label{Fig7}The framed arc $a$ and the dual arc $a^{\ast}$.}
\end{figure}

\begin{figure}
\resizebox{.5\textwidth}{!}{\includegraphics{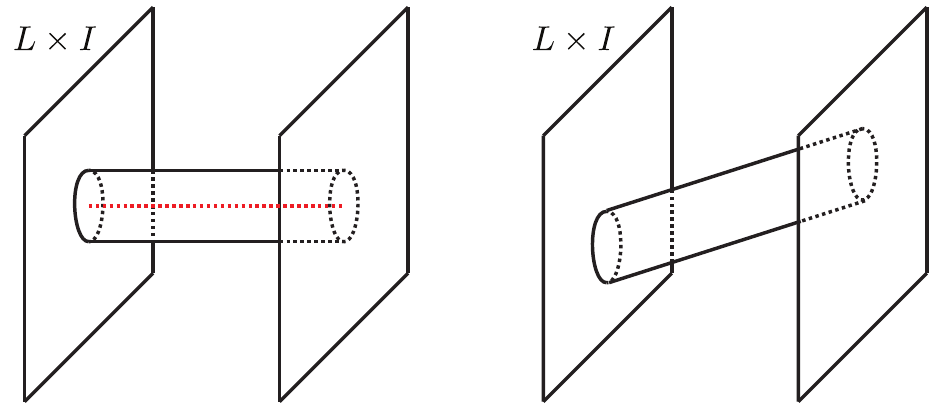}}
\caption{\label{Fig8}The surface $\overline{S_{a}}\circ S_{a}$ and its slight
perturbation along the cylinder part}
\end{figure}

The composition $\overline{S_{a}}\circ S_{a}$ is then, topologically,
a ``cylinder'' attached to $L\times I$, as shown in the left side
of Figure \ref{Fig8}. Now consider perturbing the cylinder part of
our cobordism $\overline{S_{a}}\circ S_{a}$, so that one end of the
cylinder part lies ``below'' the other end. That gives another decomposition
of $\overline{S}_{a}\circ S_{a}$, as follows:
\begin{itemize}
\item Saddle move from $L$ to $L\sqcup U$, where the unknot component
$U$ is created at one end of the arc $a$.
\item Isotopy of the component $U$, along the arc $a$. This moves $U$
to the other end of $a$.
\item Saddle move from $L\sqcup U$ to $L$.
\end{itemize}
Note that, in terms of movies of links, one can write the above decomposition
as drawn in the right of Figure \ref{Fig9}.

\begin{figure}
\resizebox{.5\textwidth}{!}{\includegraphics{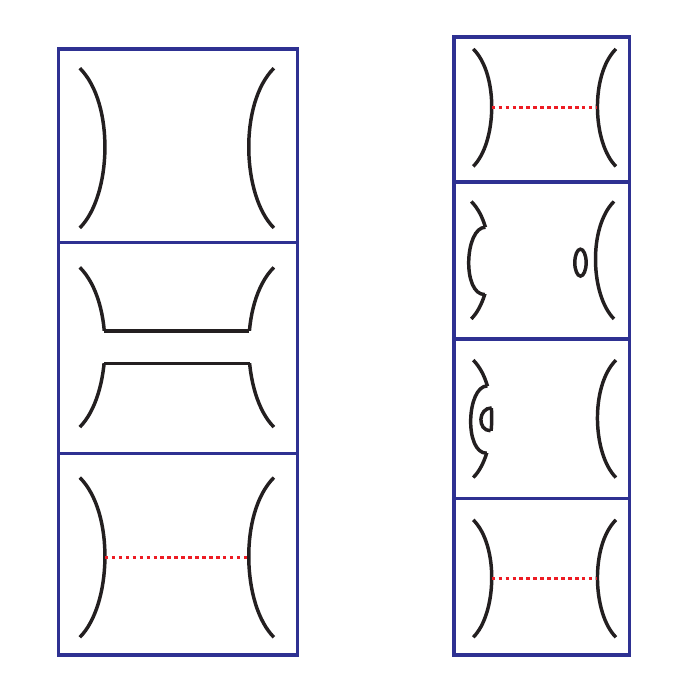}}
\caption{\label{Fig9}The movie on the left represents the saddle along $a$
followed by the saddle along $a^{\ast}$. The movie on the right represents
our new decomposition of $\overline{S_{a}}\circ S_{a}$. Here, $U$
is the unknot component appearing in the middle of the right movie
that is isotoped along $a$.}
\end{figure}

\subsection{Weak neck-passing relation}

Consider the 2-component unlink $U_{2}$. Then we can consider an
isotopy $\phi=\{\phi_{t}\}$ from $U_{2}=A\sqcup B$ to itself, defined
by moving one of its components, say $A$, around the other component
$B$, as shown in Figure \ref{Fig2}. This gives a link cobordism
$S_{\phi}$ from $U_{2}$ to itself, as follows:
\[
S_{\phi}=\bigcup_{t\in[0,1]}\phi_{t}(U_{2})\times\{t\}\subset S^{3}\times[0,1].
\]
 So, given any link TQFT $F$, we have a map $F(S_{\phi})\in\text{Aut}(F(U_{2}))$.
We consider the following relation for multiplicative link TQFTs:
\begin{description}
\item [{Weak$\ $neck-passing$\ $relation}] Let $F(U_{2})\simeq F(A)\otimes F(B)$
be the isomorphism given by the multiplicativity of $F$. Then for
any element $a\in F(U_{2})$ of the form $a=x\otimes u$, where $u$
is the unit in the Frobenius algebra $F(B)$, we have $F(S_{\phi})(a)=a$.
\end{description}
\begin{figure}
\resizebox{.3\textwidth}{!}{\includegraphics{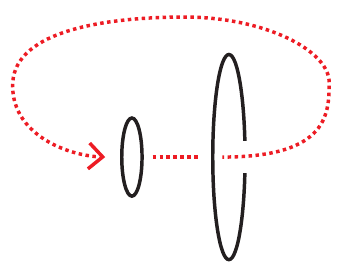}}
\caption{\label{Fig2}The ``go-around'' isotopy $\phi$ from $U_{2}$ to
itself.}
\end{figure}

We now prove that any multiplicative TQFT $F$ of (oriented) links
in $S^{3}$ satisfies the weak neck-passing relation. Consider the
birth $B_{1}$ of the component $B$, as shown in Figure \ref{Fig1}.
Then $S_{\phi}\circ B_{1}$ is isotopic to $B_{2}$. But since we
are working with links in $S^{3}$, not $\mathbb{R}^{3}$, we know
that $B_{1}$ and $B_{2}$ are isotopic by isotoping $B_{1}$ across
the point at infinity. So we have 
\[
F(B_{1})=F(B_{2})=(\text{birth map on }B\text{ component}).
\]
 Hence we get 
\[
F(S_{\phi})(x\otimes u)=F(S_{\phi})(F(B_{1})(x))=F(S_{\phi}\circ B_{1})(x)=F(B_{2})(x)=x\otimes u.
\]
 Therefore the weak neck-passing relation holds for $F$.

\begin{figure}
\resizebox{.3\textwidth}{!}{\includegraphics{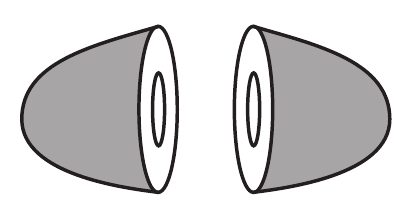}}
\caption{\label{Fig1}Two ``birth cobordisms'' $B_{1}$ and $B_{2}$ of a
component in $U_{2}$.}
\end{figure}

\subsection{Unknotting a ribbon concordance}

Let $C$ be a ribbon concordance from a knot $K\subset S^{3}$ and
$F$ be a multiplicative TQFT of (oriented) links in $S^{3}$. Then
$C$ can be decomposed as $n$ births of new unknot components $U_{1},\cdots,U_{n}$
followed by saddles along framed arcs $a_{i}$ which connect $K$
with $U_{i}$. Then $\bar{C}\circ C$ is a composition of the following
four types of cobordisms:
\begin{itemize}
\item Births of $U_{1},\cdots,U_{n}$;
\item Saddles along $a_{1},\cdots,a_{n}$;
\item Saddles along the dual arcs $b_{1},\cdots,b_{n}$, where $b_{i}$
is dual to $a_{i}$;
\item Deaths of $U_{1},\cdots,U_{n}$.
\end{itemize}
But we can see that $\bar{C}\circ C$ also admits another decomposition
into elementary cobordisms, using the observations we made in subsection
\ref{subsec-saddle}. In particular, it can be realized as follows
(as in Figure \ref{Fig3}):
\begin{itemize}
\item Births of $U_{1},\cdots,U_{n}$;
\item Saddles along arcs $e_{1},\cdots,e_{n}$, where the endpoints of $e_{i}$
are given by the two points in $\partial(\nu(K\cap a_{i}))\cap K$;
\begin{itemize}
\item Note that this move creates a new set $U_{1}^{\prime}\cdots,U_{n}^{\prime}$
of unknot components.
\end{itemize}
\item Isotopy of each $U_{i}^{\prime}$ along the framed arc $a_{i}$;
\item Saddles between each pair $U_{i}$ and $U_{i}^{\prime}$, so that
they merge into one unknot $U_{i}$;
\item Deaths of $U_{1},\cdots,U_{n}$.
\end{itemize}
\begin{figure}
\resizebox{.5\textwidth}{!}{\includegraphics{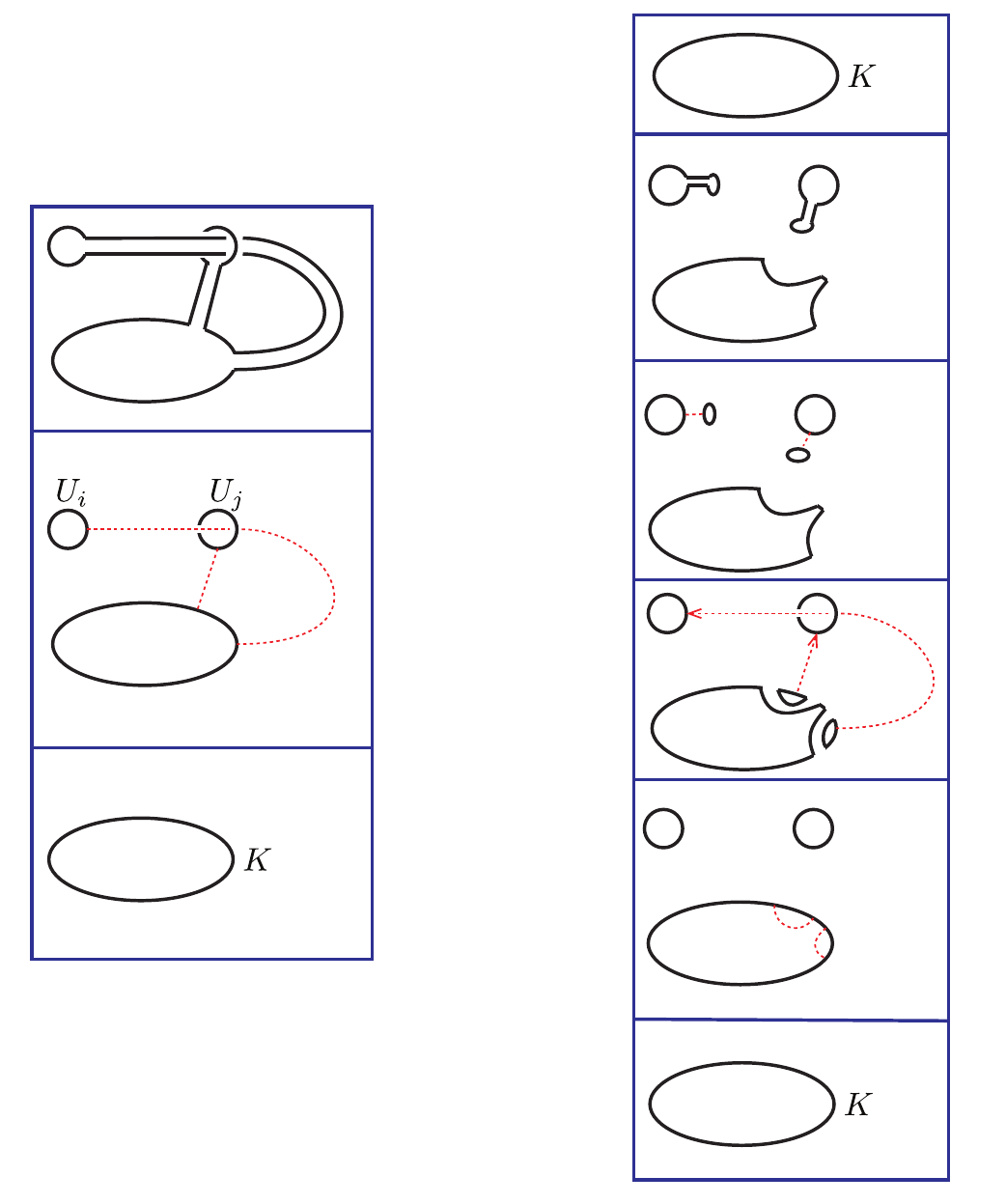}}
\caption{\label{Fig3}The movie on the left represents $C$, and the movie
on the right represents $\bar{C}\circ C$. Here, the dotted red lines
denote the framed arcs $a_{i}$, and the dotted red arrows denote
the path along which we isotope the newly created unknot components
$U_{i}^{\prime}$.}
\end{figure}

Choose a set of pairwise disjoint disks $\{D_{1},\cdots,D_{n}\}$,
each of which is disjoint from $K$, such that $\partial D_{i}=U_{i}$
for each $i$. Then we can consider the number $n(C)$, defined as
follows:
\[
n(C)=\sum_{i,j}|a_{i}\cap D_{j}|,
\]
 assuming that all intersections between arcs $a_{i}$ and disks $D_{j}$
are transverse. From now on, we will apply an induction on $n(C)$
to prove Theorem \ref{MainThm}; note that $n(C)$ only depends on
$C$ and the choice of $U_{1},\cdots,U_{n}$ and $D_{1},\cdots,D_{n}$,
and is always a nonnegative integer.

\subsubsection{The base case}

We first consider the base case, which is the case when $n(C)=0$.
Consider the sub-cobordism $S$ of $\bar{C}\circ C$, defined as the
composition of the following elementary cobordisms:
\begin{itemize}
\item Saddles along $e_{1},\cdots,e_{n}$, so that a new set $U_{1}^{\prime}\cdots,U_{n}^{\prime}$
of unknot components is created;
\item Isotopies of each $U_{i}^{\prime}$ along the framed arc $a_{i}$.
\end{itemize}
Also, consider the cobordism $S_{0}$ from an unknot $U$ to the empty
link, defined as the composition of the following elementary cobordisms:
\begin{itemize}
\item Birth of a new unknot component $U^{\prime}$;
\item Saddle between $U$ and $U^{\prime}$, so that they merge into an
unknot $U$;
\item Death of $U$.
\end{itemize}
A figure depicting the cobordisms $S$ and $S_{0}$ is drawn in Figure
\ref{Fig5}.

\begin{figure}
\resizebox{.5\textwidth}{!}{\includegraphics{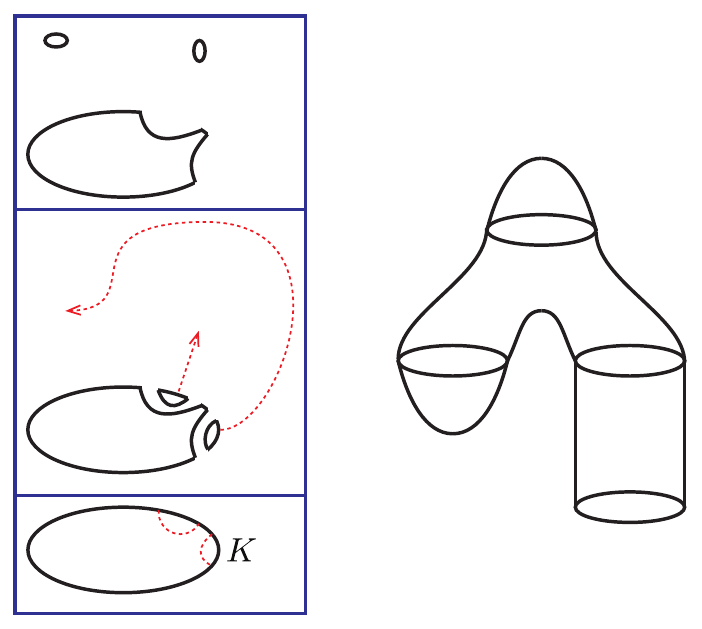}}
\caption{\label{Fig5}A movie for the cobordism $S$(left) and a figure representing
the cobordism $S_{0}$(right). Again, the dotted red lines denote
the framed arcs $a_{i}$, and the dotted red arrows denote the path
along which we isotope the newly created unknot components $U_{i}^{\prime}$. }
\end{figure}

Then, by assumption, the arcs $a_{i}^{\prime}$ never pass through
the disks $D_{j}$, so we have an isotopy 
\[
\bar{C}\circ C\sim((K\times I)\sqcup S_{0}\sqcup\cdots\sqcup S_{0})\circ S.
\]
 But $S_{0}$ is isotopic to the death cobordism $D$, so we get an
isotopy 
\[
\bar{C}\circ C\sim((K\times I)\sqcup D\sqcup\cdots\sqcup D)\circ S.
\]
 Now the cobordism $((K\times I)\sqcup D\sqcup\cdots\sqcup D)\circ S$
is isotopic to the cylinder $K\times I$. Thus we get 
\[
\bar{C}\circ C\sim K\times I.
\]
 Therefore we have 
\[
F(\bar{C})\circ F(C)=F(\bar{C}\circ C)=F(K\times I)=\text{id}.
\]
 This proves the base case of Theorem \ref{MainThm}.

\subsubsection{Inductive step, when $F$ is associative}

Now suppose that $n(C)>0$. Then we can isotope $\bar{C}\circ C$
so that the map 
\[
T:\bigcup_{i,j}(a_{i}\cap D_{j})\hookrightarrow\bar{C}\circ C\hookrightarrow S^{3}\times I\twoheadrightarrow I
\]
 is injective, i.e. all intersection points $a_{i}\cap D_{j}$ occur
in ``distinct times''. Choose a point $p\in a_{i}\cap D_{j}$ at
which the function $T$ takes its minimum, and construct another ribbon
concordance $C^{\prime}$, as shown in Figure \ref{Fig4}, using the
same saddle-arcs $a_{k}$ for all $k\ne i$ but replacing $a_{i}$
by a new framed arc $a_{i}^{\prime}$. Here, $a_{i}^{\prime}$ should
satisfy the following conditions.
\begin{itemize}
\item $a_{i}\cap a_{i}^{\prime}=\emptyset$.
\item $a_{i}\cup a_{i}^{\prime}$ is isotopic to a $0$-framed meridian
of $U_{j}$ which intersects once with $D_{j}$ but does not intersect
with any other $D_{k}$ nor the knot $K$. 
\end{itemize}
\begin{figure}
\resizebox{.5\textwidth}{!}{\includegraphics{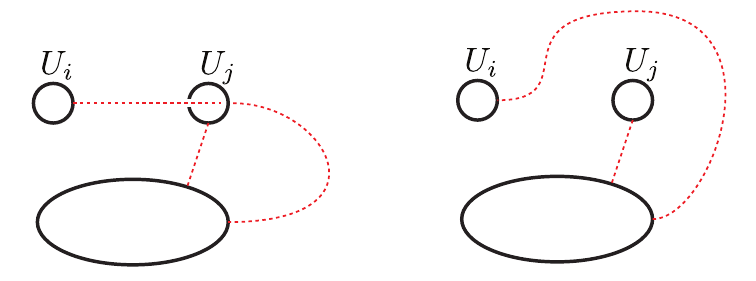}}
\caption{\label{Fig4}The given ribbon concordance $C$(left) and the new ribbon
concordance $C^{\prime}$(right). Again, dotted red lines are the
framed arcs along which we perform saddle moves. }
\end{figure}

Then the concordances $\overline{C^{\prime}}\circ C^{\prime}$ and
$\bar{C}\circ C$ differ in the following way. In the movie of $\bar{C}\circ C$
drawn in Figure \ref{Fig3}, denote the composition of the first two
steps, i.e. births of $U_{1},\cdots,U_{n}$ followed by saddle moves
from $K$ to $K\sqcup U_{1}^{\prime}\sqcup\cdots\sqcup U_{n}^{\prime}$,
by $S_{1}$, and denote the composition of the rest by $S_{2}$. Furthermore,
denote the self-concordance of $K\sqcup(\sqcup U_{k})\sqcup(\sqcup U_{k}^{\prime})$
given by the ``neck-passing'' of $U_{i}^{\prime}$ through $U_{j}$
by $S_{ij}$. Then we have 
\begin{align*}
\bar{C}\circ C & =S_{2}\circ S_{1},\\
\overline{C^{\prime}}\circ C^{\prime} & =S_{2}\circ\overline{S_{ij}}\circ S_{1},
\end{align*}
 where we have assumed without loss of generality that $a_{i}$ and
$D_{j}$ intersect positively at $p$. Now choose any $x\in F(K)$.
Then, under the multiplicativity isomorphism 
\[
F(K\sqcup(\sqcup U_{k})\sqcup(\sqcup U_{k}^{\prime}))\xrightarrow{\sim}\left(F(K)\otimes\left(\bigotimes_{k\ne j}F(U_{i})\right)\otimes\left(\bigotimes_{k\ne i}F(U_{i}^{\prime})\right)\right)\otimes F(U_{i}^{\prime})\otimes F(U_{j}),
\]
 we have $F(S_{1})(x)=\sum_{m}x_{m}^{\prime}\otimes y_{m}\otimes u$
for some $x_{m}^{\prime}\in F(K)\otimes\left(\bigotimes_{k\ne j}F(U_{i})\right)\otimes\left(\bigotimes_{k\ne i}F(U_{i}^{\prime})\right)$
and $y_{m}\in F(U_{i}^{\prime})$ by associativity. Then, by the functoriality
of $F$ and the weak neck-passing relation, we have 
\[
F(S_{ij})(x_{m}^{\prime}\otimes y_{m}\otimes u)=x_{m}^{\prime}\otimes F(S_{\phi})(y_{m}\otimes u)=x_{m}^{\prime}\otimes y_{m}\otimes u.
\]
 Hence $F(S_{ij}\circ S_{1})(x)=F(S_{ij})(F(S_{1})(x))=F(S_{1})(x)$,
which implies $F(\overline{S_{ij}}\circ S_{1})=F(S_{1})$ by functoriality.
But then we have 
\[
F(\overline{C^{\prime}}\circ C^{\prime})=F(S_{2})\circ F(\overline{S_{ij}}\circ S_{1})=F(S_{2})\circ F(S_{1})=F(\bar{C}\circ C).
\]
 Also, we have $n(C^{\prime})=n(C)-1$ by the construction of $C^{\prime}$.
Therefore we have $F(\bar{C})\circ F(C)=\text{id}$ by induction on
$n(C)$; this proves Theorem \ref{MainThm} in the associative case.

\subsubsection{The case when $F$ is Khovanov-like}

We now consider the case when $F$ is Khovanov-like, but not necessarily
associative. Then the proof in the associative case cannot be applied
directly, since the observation $F(S_{1})(x)=\sum_{m}x_{m}^{\prime}\otimes y_{m}\otimes u$
relies on the associativity of $F$. However we can still prove the
same observation using our new assumption.

Since $F$ is now Khovanov-like, the unit $u\in F(\text{unknot})$
spans the kernel of the death map $\epsilon$ by definition. Under
the same notation as used in the proof of the associative case, consider
the death cobordism $E_{j}$ of the link $K\sqcup(\sqcup U_{k})\sqcup(\sqcup U_{k}^{\prime})$.
Then the cobordism $E_{j}\circ S_{1}$ contains a closed sphere which
bounds a 3-ball in $S^{3}\times I$. Thus, by the multiplicativity
of $F$, we get $F(E_{j})\circ F(S_{1})=0$. But again by the multiplicativity,
under the isomorphism 
\[
F(K\sqcup(\sqcup U_{k})\sqcup(\sqcup U_{k}^{\prime}))\xrightarrow{\sim}\left(F(K)\otimes\left(\bigotimes_{k\ne j}F(U_{i})\right)\otimes\left(\bigotimes_{k\ne i}F(U_{i}^{\prime})\right)\right)\otimes F(U_{i}^{\prime})\otimes F(U_{j}),
\]
 the map $F(E_{j})$ is given by $\text{id}\otimes\text{id}\otimes\epsilon$.
Therefore the observation $F(S_{1})(x)=\sum_{m}x_{m}^{\prime}\otimes y_{m}\otimes u$
is still holds in this case, and the rest of the proof is the same.
This proves Theorem \ref{MainThm} in the Khovanov-like case.

\subsection{Proofs of the corollaries \ref{MainCor} and \ref{MainCor2}}

Finally, using Theorem \ref{MainThm}, which was proved in the last
subsection, we can now prove the Corollaries \ref{MainCor} and \ref{MainCor2}.
\begin{proof}[Proof of Corollary \ref{MainCor}]
All conic strong Khovanov-Floer theories are multiplicative and Khovanov-like.
So the corollary holds for all conic strong Khovanov-Floer theories.

For the case of Khovanov-Rozansky homology, it is proven in \cite{ehrig2017functoriality}
that Khovanov-Rozansky homology is a TQFT of links in $\mathbb{R}^{3}$.
Moreover, that result was upgraded in \cite{morrison2019invariants},
which proves that it is actually a TQFT of links in $S^{3}$. Since
Khovanov-Rozansky homology is multiplicative and associative by its
definition, we see that it induces injective maps for ribbon concordances
by Theorem \ref{MainThm}.
\end{proof}
\begin{proof}[Proof of Corollary \ref{MainCor2}]
If two vector spaces $V,W$ over $\mathbb{F}$ admit linear injections
$V\rightarrow W$ and $W\rightarrow V$, then $V\simeq W$.
\end{proof}

\section{$\mathbb{Z}$-grading and the neck passing relation}

\subsection{Nicely graded conic strong Khovanov-Floer theories}

Some strong Khovanov-Floer theories come with a $\mathbb{Z}$-grading.
We will say that a conic strong Khovanov-Floer theory $F$ is \textit{nicely
graded} if it carries a $\mathbb{Z}$-grading such that the cobordism
maps for $F$ are degree-preserving up to some degree shift, and that
$F(\text{unknot})$ is not concentrated in one grading. In such cases,
we can get a relation which is much stronger than the weak neck-passing
relation. 

Note that we have an isomorphism of graded vector spaces 
\[
F(\text{unknot})\simeq\mathbb{F}[X]/(X^{2}),
\]
 which maps to unit $u$ to $1$; the counit $\epsilon:F(\text{unknot})\rightarrow\mathbb{F}$
is given by sending $1$ to $0$ and $X$ to $1$. With respect to
such an identification, the assumption that $F$ is nicely graded
is equivalent to assuming that the unit $1$ and the element $X$
lie in different gradings.

Consider the two-component unknot $U_{2}=A\sqcup B$ and define $S_{\phi}$
as in the previous section. Then we have the following theorem.
\begin{thm}[Neck-passing relation]
Let $F$ be a nicely graded conic strong Khovanov-Floer theory. Then
$F(S_{\phi})=\text{id}$.
\end{thm}

\begin{proof}
Consider the birth cobordism $B_{A}$ for the component A, i.e. cobordism
given by 
\[
B_{A}=(\text{birth for }A)\cup(\text{cylinder for }B).
\]
 Then $B_{A}$ is an oriented link cobordism from $B$ to $U_{2}$,
and we have $F(B_{A})(1)=X\otimes1$, where we are taking the identification
\[
F(U_{2})\simeq\mathbb{F}[X]/(X^{2})\otimes\mathbb{F}[X]/(X^{2}),
\]
 and the first component in the tensor product corresponds to the
component $A$ of $U_{2}$. But then $S_{\phi}\circ B_{A}$ is isotopic
to $B_{A}$, so we have 
\[
F(S_{\phi})(1\otimes X)=F(S_{\phi})(F(B_{A})(1))=F(S_{\phi}\circ B_{A})(1)=F(B_{A})(1)=1\otimes X,
\]
 so the map $F(S_{\phi})$ fixes $1\otimes X$. Similarly, we can
see that $F(S_{\phi})$ also fixes $1\otimes1$.

By the weak neck-passing relation, we already know that $F(S_{\phi})$
also fixes $X\otimes1$. Thus it remains to prove that $F(S_{\phi})(X\otimes X)=X\otimes X$.
By the assumption that $F$ is nicely graded, we know that the $2\cdot\mathbf{gr}(X)$-graded
piece of $F(U_{2})$ has rank $1$, generated by $X\otimes X$. Also,
we know that the grading shift of $F(S_{\phi})$ is $0$ by the weak
neck-passing relation. Thus we already know that $F(S_{\phi})(X\otimes X)=c\cdot X\otimes X$
for some scalar $c\in\mathbb{F}$. However, using the ``upside-down''
version of our argument, we can prove that $c=1$ as follows:
\[
c=F(\overline{B_{A}})(c\cdot X\otimes X)=F(\overline{B_{A}}\circ S_{\phi})(X\otimes X)=F(\overline{B_{A}})(X\otimes X)=X.
\]
 Therefore we deduce that $F(S_{\phi})=\text{id}$.
\end{proof}
Using the above theorem, we can actually prove a stronger statement,
although it will not be used in this paper. Let $L_{0}$ be a link
and $L=L_{0}\sqcup U$, where $U$ is an unknot. Choose any component
$K\subset L_{0}$, and a meridian $m$ of $K$. Then we can consider
the self-isotopy $\phi_{L_{0},K}$ of $L$ defined by moving $U$
along $m$. As in the neck-passing relation, we can consider the link
cobordism $S_{L_{0},K}$, defined as 
\[
S_{L_{0},K}=\bigcup_{t\in I}(\phi_{t}(L)\times\{t\})\subset S^{3}\times I.
\]
 Then, for any link TQFT $F$, we can consider the morphism $F(S_{L_{0},K})$. 
\begin{cor}[Strong neck-passing relation]
Let $F$ be a nicely graded conic strong Khovanov-Floer theory. Then
for any choice of $L_{0}$ and $K$, the map $F(S_{L_{0},K})$ is
the identity.
\end{cor}

\begin{proof}
Consider the saddle cobordism with respect to an arc $a$ satisfying
the following conditions:
\begin{itemize}
\item $a$ is interior-disjoint from $L$, and its boundary points $p,q$
lie on $K$, at which $a$ is transverse to $K$.
\item Taking saddle of $L\cup m$, where $m$ is a meridian of $K$, along
$a$, gives the link $L\cup\text{(Hopf link})$.
\end{itemize}
Then the saddle cobordism $S_{a}$ from $L$ to $L\cup\text{unknot}$
admits a left inverse, which is the death cobordism of the newly created
unknot component. Thus $F(S_{a})$ is injective. 

Now consider the following diagram.
\[
\xymatrix{L\ar[rr]^{S_{a}}\ar[d]^{S_{L_{0},K}} &  & L\cup U\ar[d]^{S_{L_{0}\cup U,U}}\\
L\ar[rr]^{S_{a}} &  & L\cup U
}
\]
 Since $S_{L_{0}\cup U,U}\circ S_{a}$ is isotopic to $S_{a}\circ S_{L_{0},K}$,
we get the following commutative square. Note that the square on the
right side is due to the multiplicativity of $F$.
\[
\xymatrix{F(L)\ar[rr]^{F(S_{a})}\ar[d]^{F(S_{L_{0},K})} &  & F(L\cup U)\ar[d]^{F(S_{L_{0}\cup U,U})}\ar[rr]^{\simeq} &  & F(L)\otimes_{R}F(U)\ar[d]^{\text{id}\otimes F(S_{U,U})}\\
F(L)\ar[rr]^{F(S_{a})} &  & F(L\cup U)\ar[rr]^{\simeq} &  & F(L)\otimes_{R}F(U)
}
\]
 But we already know that $F(S_{U,U})$ is the identity. Therefore,
by the injectivity of $F(S_{a})$, we deduce that $F(S_{L_{0},K})=\text{id}$.
\end{proof}

\subsection{Knot Floer homology}

The above proof cannot be used directly to prove that ribbon concordances
induce injective maps between knot Floer homology, because of the
following reasons:
\begin{itemize}
\item Knot Floer homology is not a TQFT of links and link cobordisms, but
rather a TQFT of decorated links and decorated link cobordisms.
\item Knot Floer homology is a reduced theory, i.e. we have a natural splitting
\[
HFK^{\circ}(L_{1}\sqcup L_{2},P_{1}\sqcup P_{2})\simeq HFK^{\circ}(L_{1},P_{1})\otimes HFK^{\circ}(L_{2},P_{2})\otimes V
\]
 where $\circ$ is either hat or minus flavor and and $V=\mathbb{F}^{2}$.
\end{itemize}
Here, we recall that a decorated link is a link together with $z$-basepoints
and $w$-basepoints which occur in alternating way, so that each component
has at least two basepoints. Also, decorated link cobordism is a splitting
of a given cobordism into two subsurfaces such that one contains all
$z$-basepoints and the another contains all $w$-basepoints. For
more details, see \cite{juhasz2018computing} and \cite{zemke2019link}.

Now consider the 2-component unknot $U_{2}$, together with the decoration
$P$, so that each component of $U_{2}$ has one $z$-basepoint and
$w$-basepoint. Then we can construct a decoration $P_{\phi}$ on
the ``go-around'' cobordism $S_{\phi}$ from $U_{2}$ to itself,
so that for each cylinder component $C\subset S_{\phi}$, the decoration
$P_{\phi}|_{C}$ is given by Figure \ref{Fig6}.

\begin{figure}
\resizebox{.3\textwidth}{!}{\includegraphics{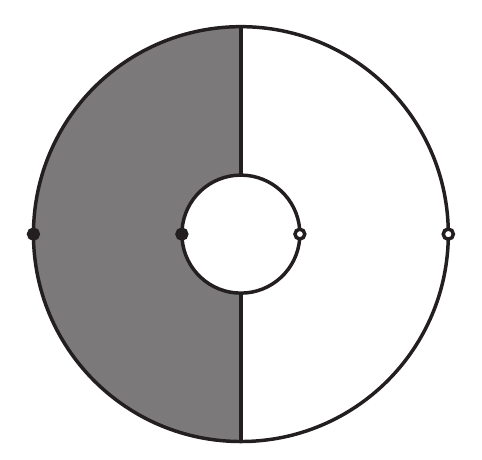}}
\caption{\label{Fig6}The decoration $P_{\phi}$ on the cylinder component
$C$.}
\end{figure}

Of course, the decoration $P_{\phi}$ on $S_{\phi}$ is not uniquely
defined. However we can choose one anyway, which will give us a map
\[
HFK^{\circ}(S_{\phi},P_{\phi})\,:\,HFK^{\circ}(U_{2},P)\rightarrow HFK^{\circ}(U_{2},P),
\]
 and this map is an automorphism because the decorated cobordism $(S_{\phi},P_{\phi})$
obviously has an inverse.

Now, when $\circ=\text{hat}$, then we have 
\[
\widehat{HFK}(U_{2},P)\simeq\mathbb{F}[\frac{1}{2}]\oplus\mathbb{F}[-\frac{1}{2}],
\]
 and when $\circ=\text{minus}$, we have 
\[
HFK^{-}(U_{2},P)\simeq\mathbb{F}[U][\frac{1}{2}]\oplus\mathbb{F}[U][-\frac{1}{2}].
\]
 In either case, the only Maslov grading-preserving automorphism of
$HFK^{\circ}(U_{2},P)$ , where $\circ$ is either the minus or hat
flavor, is the identity. Furthermore, the only automorphism of $\widehat{HFK}(U_{2},P)$
which has a constant grading shift is the identity, which has zero
grading shift. Hence, in either hat-flavor or minus-flavor, the grading
shift of $HFK^{\circ}(S_{\phi},P_{\phi})$ is zero, and thus we have
\[
HFK^{\circ}(S_{\phi},P_{\phi})=\text{id}.
\]
 Therefore, by repeating our proof in the previous section, but now
using the splitting formula 
\[
HFK^{\circ}(L_{1}\sqcup L_{2},P_{1}\sqcup P_{2})\simeq HFK^{\circ}(L_{1},P_{1})\otimes HFK^{\circ}(L_{2},P_{2})\otimes V
\]
 for knot Floer homology, together with the splitting formula for
disjoint unions of cobordisms, given by 
\[
HFK^{\circ}(S_{1}\sqcup S_{2},P_{S_{1}}\sqcup P_{S_{2}})=HFK^{\circ}(S_{1},P_{S_{1}})\otimes HFK^{\circ}(S_{2},P_{S_{2}})\otimes\text{id}_{V},
\]
 we deduce that every ribbon concordance induces an injective map
between $HFK$, in both hat- and minus-flavor.
\begin{rem}
Using the arguments in the last section to knot Floer homology, we
can easily see that neck-passing relation and strong neck-passing
relation hold for knot Floer homology. Of course we should choose
a decoration on our link cobordisms as in Figure \ref{Fig6}.
\end{rem}

\section{$\widehat{HF}$ of the branched double cover}

\subsection{An alternative proof of the injectivity of $HF^{\circ}\circ\Sigma$
for hat- and minus-flavors}

Consider the Heegaard Floer homology of the double branched cover,
defined as the link TQFT 
\[
L\subset S^{3}\mapsto HF^{\circ}(\Sigma(L)),
\]
 where we take the flavor $\circ$ to be either hat or minus. Then
the resulting TQFT satisfies functoriality for link cobordisms, defined
by 
\[
\text{cobordism }S\mapsto\text{map }F_{\Sigma(S)}^{\circ},
\]
 but this carries a similar problem as in the case of knot Floer homology. 

To be precise, the problem is the following. Although the assignment
$L\mapsto HF^{\circ}(\Sigma(L)\sharp(S^{1}\times S^{2}))$ is a (unreduced)
conic strong Khovanov-Floer theory, the assignment $L\mapsto HF^{\circ}(\Sigma(L))$
is not, since it satisfies a reduced version of multiplicativity 
\[
HF^{\circ}(\Sigma(L_{1}\sqcup L_{2}))\simeq HF^{\circ}(\Sigma(L_{1}))\otimes HF^{\circ}(\Sigma(L_{2}))\otimes V,
\]
 where the isomorphism is again natural with respect to cobordism
maps. However, since we have 
\[
HF^{\circ}(\Sigma(U_{2}))\simeq HF^{\circ}(S^{1}\times S^{2}),
\]
 the only degree-preserving automorphism of $HF^{\circ}(\Sigma(U_{2}))$
is the identity. Thus, using the same argument used in the knot Floer
case, we see that $\widehat{HF}\circ\Sigma$ satisfies the neck-passing
relation. Therefore, for any ribbon concordance $C:K_{1}\rightarrow K_{2}$,
the cobordism map $F_{\Sigma(C)}^{\circ}$ is injective, as already
shown in \cite{lidman2019heegaard} using a different method.

\subsection{Involutions on $\widehat{HF}(\Sigma(K))$}

Since $F_{\Sigma(\bar{C})}^{\circ}F_{\Sigma(C)}^{\circ}=\text{id}$
by the neck-passing relation, we actually know that $F_{\Sigma(C)}^{\circ}$
induces an inclusion of $HF^{\circ}(\Sigma(K_{1}))$ in $HF^{\circ}(\Sigma(K_{2}))$
in a way that it becomes a direct summand. This gives a very strong
restriction on the deck transformation action(which we will denote
as $\tau$) and the $\iota$-involution(which arises naturally in
the construction of involutive Floer homology in \cite{hendricks2017involutive})
on $HF^{\circ}(\Sigma(K_{1}))$ when $K_{2}$ satisfies some nice
conditions.

We briefly recall the definition of the two involutions $\tau$ and
$\iota$. By the naturality of Heegaard Floer theory, due to Juhasz
and Thurston\cite{juhasz2012naturality}, for any 3-manifold $M$
with a basepoint $z$, the pointed mapping class group $\mathbf{Mod}(M,z)$
acts on $\widehat{HF}(M)$. When $M=\Sigma(K)$ and $z\in K$, the
deck transformation of $\Sigma(K)\rightarrow S^{3}$ fixes $z$, thus
gives a $\mathbb{Z}_{2}$-action $\tau$ on $\widehat{HF}(M)$.

The involution $\iota$ is defined in a much more subtle way. Choose
any Heegaard diagram $(\Sigma,\boldsymbol{\alpha},\boldsymbol{\beta},z)$
representing $\Sigma(K)$. Then we have the identity map 
\[
\text{id}:\widehat{CF}(\Sigma,\boldsymbol{\alpha},\boldsymbol{\beta},z)\rightarrow\widehat{CF}(\bar{\Sigma},\boldsymbol{\beta},\boldsymbol{\alpha},z),
\]
 and since both $(\Sigma,\boldsymbol{\alpha},\boldsymbol{\beta},z)$
and $(\bar{\Sigma},\boldsymbol{\beta},\boldsymbol{\alpha},z)$ represent
$\Sigma(K)$, we have a naturality map 
\[
f:\widehat{CF}(\bar{\Sigma},\boldsymbol{\beta},\boldsymbol{\alpha},z)\rightarrow\widehat{CF}(\Sigma,\boldsymbol{\alpha},\boldsymbol{\beta},z),
\]
 which is defined uniquely up to chain homotopy. Then $f\circ\text{id}$
is a homotopy involution, so the induced automorphism on $\widehat{HF}(\Sigma(K))$
is a uniquely determined involution, which we denote as $\iota$.

As shown in \cite{alfieri2019connected}, the behaviors of two involutions
$\tau$ and $\iota$ of $\widehat{HF}(\Sigma(K))$ are a bit different:
sometimes they are identical, whereas sometimes they are not. To be
precise, we know the following:
\begin{itemize}
\item When $K$ is quasi-alternating, then $\tau$ and $\iota$ are both
trivial.
\item When $K$ is an odd torus knot, then $\tau$ is trivial, but $\iota$
is nontrivial in general.
\item When $K$ is a Montesinos knot, then $\tau=\iota$.
\end{itemize}
Note that, in the Montesinos case, there exists a $\iota$-invariant
basis of $\widehat{HF}(\Sigma(K))$ such that the action of $\iota$
leaves exactly one basis element fixed, as shown in \cite{dai2017involutive}.
Furthermore, it is straightforward to see that $\tau$ and $\iota$
always commute, i.e. $\tau\circ\iota=\iota\circ\tau$. Using these
results, we can now prove Theorem \ref{MainApp}.
\begin{proof}[Proof of Theorem \ref{MainApp}]
Suppose that $K_{0}$ is ribbon concordant to $K_{1}$ by a ribbon
concordance $C$, and let $\sigma$ denote an involution, which is
either $\tau$, $\iota$, or $\tau\circ\iota$. Then the involution
$\sigma$ gives $\mathbb{F}[\mathbb{Z}_{2}]$-module structures on
$\widehat{HF}(\Sigma(K_{0}))$ and $\widehat{HF}(\Sigma(K_{1}))$.
Furthermore, the cobordism map $\hat{F}_{\Sigma(C)}$ commutes with
$\sigma$ (since it commutes with both $\tau$ and $\iota$; see \cite{alfieri2019connected}
and \cite{hendricks2017involutive}), and $\hat{F}_{\Sigma(\bar{C})}\hat{F}_{\Sigma(C)}=\text{id}$,
so $\widehat{HF}(\Sigma(K_{0}))$ is a $\mathbb{F}[\mathbb{Z}_{2}]$-module
direct summand of $\widehat{HF}(\Sigma(K_{1}))$.

But it is obvious that every finitely generated $\mathbb{F}[\mathbb{Z}_{2}]$-module
$M$ can be uniquely represented as 
\[
M=\mathbb{F}^{m_{M}}\oplus(\mathbb{F}\cdot v\oplus\mathbb{F}\cdot\sigma(v))^{n_{M}},
\]
 so that if an $\mathbb{F}[\mathbb{Z}_{2}]$-module $M$ is a direct
summand of $N$, then $m_{M}\le m_{N}$ and $n_{M}\le n_{N}$. This
proves the theorem.
\end{proof}
\bibliographystyle{amsalpha}
\bibliography{ribconc}

\end{document}